\documentclass[a4paper, reqno, 11pt]{amsart}

\usepackage[sc, noBBpl]{mathpazo}	
\usepackage{amssymb, amsmath, amsthm}

\usepackage[T1]{fontenc}
\usepackage[utf8]{inputenc}
\usepackage{graphicx}
\usepackage{comment}
\usepackage[linkbordercolor={0.8 0.9 0.9}, citebordercolor={0.85 0.85 0.85}]{hyperref}
\usepackage[normalem]{ulem}

\usepackage{amsmath, amsthm, amssymb} 
\usepackage[all]{xy}
\usepackage[usenames,dvipsnames]{color}

\numberwithin{equation}{section}

\theoremstyle{plain}
\newtheorem{theorem}{Theorem}[section]
\newtheorem{proposition}[theorem]{Proposition}
\newtheorem{lemma}[theorem]{Lemma}
\newtheorem{corollary}[theorem]{Corollary}
\theoremstyle{definition}
\newtheorem{definition}[theorem]{Definition}
\newtheorem{example}[theorem]{Example}
\theoremstyle{remark}
\newtheorem{remark}[theorem]{Remark}

  {\par\begin{tabular}{rcl}}%
  {\end{tabular}\par}

\newcommand{\st}{\;|\;}
\newcommand{\defeq}{:=}
\newcommand{\half}{{\frac{1}{2}}}
\newcommand{\ip}[1]{\langle #1 \rangle}

\newcommand{\lie}{\mathfrak}

\newcommand{\CC}{\mathbb{C}}

\newcommand{\LL}{\mathbb{L}}
\newcommand{\NN}{\mathbb{N}}
\newcommand{\RR}{\mathbb{R}}

\newcommand{\ZZ}{\mathbb{Z}}

\newcommand{\sA}{\mathcal{A}}
\newcommand{\sB}{\mathcal{B}}
\newcommand{\sC}{\mathcal{C}}
\newcommand{\sD}{\mathcal{D}}
\newcommand{\sE}{\mathcal{E}}

\newcommand{\sH}{\mathcal{H}}

\newcommand{\sO}{\mathcal{O}}

\newcommand{\sU}{\mathcal{U}}

\newcommand{\bfmu}{{\boldsymbol{\mu}}}

\newcommand{\SU}{\mathrm{SU}}

\renewcommand{\Re}{\mathop{\mathrm{Re}}\nolimits}

\DeclareMathOperator{\End}{End}

\DeclareMathOperator{\Tr}{Tr}

\DeclareMathOperator{\Res}{Res}

\DeclareMathOperator{\id}{id}

\DeclareMathOperator{\ord}{ord}

\DeclareMathOperator{\Op}{Op}

\DeclareMathOperator{\dom}{dom}

\newcommand{\Uqt}{\sU_q(\lie{t})}
\newcommand{\Uqk}{\sU_q(\lie{k})}

\newcommand{\hit}{\!\triangleright\!}

\newcommand{\DO}{\mathrm{DO}}

\newcommand{\slot}{\,\cdot\,}
\newcommand{\CP}{\CC \mathrm{P}}
\newcommand{\algord}{\ord_\sD}

\newcommand{\dmu}{\partial_{\bfmu}}

\newcommand{\Dirac}{\slash\!\!\!\!D}

\newcommand{\PsiDO}{\mathrm{\Psi DO}}
\newcommand{\PsiDOo}{\mathrm{\Psi DO}_0}
\newcommand{\Uq}{\sU_q}

\begin{document}

\title{Regularity of twisted spectral triples and pseudodifferential calculi}
\author{Marco Matassa}
\address{Université Clermont Auvergne, Université Blaise Pascal, BP 10448, F-63000 Clermont-Ferrand, France}
\email{Marco.Matassa@math.univ-bpclermont.fr}

\author{Robert Yuncken}
\thanks{R.~Yuncken was supported by the project SINGSTAR of the Agence Nationale de la Recherche, ANR-14-CE25-0012-01.}
\address{Université Clermont Auvergne, Université Blaise Pascal, BP 10448, F-63000 Clermont-Ferrand, France}
\email{robert.yuncken@math.univ-bpclermont.fr}

\date{}

\subjclass[2010]{Primary: 58J42; Secondary 46L87, 58B32}
\keywords{Noncommutative geometry; spectral triple; local index formula; type III non-commutative geometry; quantum groups}

\maketitle

\begin{abstract}
We investigate the regularity condition for twisted spectral triples. This condition is equivalent to the existence of an appropriate pseudodifferential calculus compatible with the spectral triple. A natural approach to obtain such a calculus is to start with a twisted algebra of abstract differential operators, in the spirit of Higson. Under an appropriate algebraic condition on the twisting, we obtain a pseudodifferential calculus which admits an asymptotic expansion, similarly to the untwisted case.
We present some examples coming from the theory of quantum groups. Finally we discuss zeta functions and the residue (twisted) traces on differential operators.
\end{abstract}

\section{Introduction}

The basic structure in Connes' noncommutative geometry \cite{Connes:NCG} is a \emph{spectral triple}, consisting of a $*$-algebra $\sA$ represented on a Hilbert space $H$ and equipped with an unbounded self-adjoint operator $D$.  These must satisfy certain axioms---notably, in the standard formulation, that the commutators $[D,a]$ be bounded for all $A\in \sA$. 
However, the full power of the definition is unleashed only when one adds the additional property of \emph{regularity}: that $\sA$ and $[D,\sA]$ are in the domain of the derivation $\delta = [|D|,\slot]$ as well as all its iterates $\delta^n$.  
This is the context in which one obtains the celebrated Local Index Formula of Connes and Moscovici \cite{ConMos:local_index_formula}.  


On the other hand, once one leaves the commutative world one quickly finds that the bog standard definition of a spectral triple leaves out many interesting examples.  One new phenomenon that occurs is ``twisting'',  also called ``type III noncommutative geometry'' \cite{ConMos:twisted}.  Connes and Moscovici were motivated to study this by index theory for foliations, but a similar phenomenon arises in the study of quantum homogeneous spaces---see, \emph{e.g.}, \cite{NesTus:local_index_formula}.

A \emph{twisting} means an algebra automorphism $\theta$ of $\sA$.  We use the following notation for twisted commutators:
\[
 [a,b]_\theta := ab-\theta(b)a.
\]

\begin{definition}[Connes-Moscovici \cite{ConMos:twisted}]
\label{def:twisted_spectral_triple}
 A \emph{twisted (unital) spectral triple} is a triple\footnote{This would be a quadruple if one included the twisting $\theta$ in the data, making for some awkward terminological choices.  Connes and Moscovici use the name \emph{$\theta$-spectral triple.}} $(\sA,H,D)$ consisting of a unital $*$-algebra $\sA$ with twisting $\theta$ such that $\sA$ is represented as bounded operators on a Hilbert space $H$, together with an unbounded self-adjoint operator $D$ with compact resolvent such that $[D,a]_\theta$ is densely defined and bounded for all $a\in \sA$.
\end{definition}

Defining regularity for twisted spectral triples becomes a little awkward.  One should clearly replace the derivation $\delta= [|D|,\slot]$ with the twisted derivation 
\[
 \delta_\theta = [|D|,\slot]_\theta.
\]
But the repeated commutators $\delta_\theta^n(a)$ are not well-defined until one extends the twisting $\theta$ to each $\delta_\theta^{n-1}(\sA)$ in turn.  In \cite{ConMos:twisted}, Connes and Moscovici only  define Lipschitz regularity explicitly---considering only one application of $\delta_\theta$---but they could easily have made the following generalization.

\begin{definition}
\label{def:regular}
 A twisted spectral triple $(\sA, H, D)$ with twisting $\theta$ is \emph{regular} if there is a larger algebra  $\sB\subseteq\LL(H)$ containing both $\sA$ and $[D,\sA]_\theta$ and which is equipped with an extension of $\theta$ as a \emph{linear} isomorphism such that $\sB$ is invariant under $\delta_\theta$.
\end{definition}

\begin{remark}
Note that, unlike \cite{ConMos:twisted}, we are only requiring $\theta$ to be a linear isomorphism, not an algebra automorphism, on $\sB$.  This point is crucial for applications to quantum groups.  On the other hand, $\theta$ will generally be an automorphism on "principal symbols"---see Remark \ref{rmk:symbol_aut}.
\end{remark}

Definition \ref{def:regular} is obviously the correct generalization, but how to verify it in practice?  Higson \cite{Higson:local_index_formula}, distilling the ideas of Connes and Moscovici \cite{ConMos:local_index_formula}, pointed out that the point of entry in practice is the existence of an algebra of ``abstract differential operators" $\sD$ and a Laplace-type operator $\Delta$.  Combining the complex powers of the Laplacian and the abstract differential operators results in "abstract pseudodifferential operators" with the essential properties that one needs to run the Connes-Moscovici machine.

The abstract properties of the Laplace-type operator $\Delta$ are (a) that it satisfies elliptic estimates and (b) that its principal symbol be central.  The main point of this article is to show that, under some mild assumptions, we can replace (b) by twisted centrality and produce a twisted pseudodifferential calculus.  We also obtain some of the immediate consequences of the pseudodifferential calculus: regularity of twisted spectral triples, and residue traces of zeta-functions.

This is the appropriate framework for studying Kr\"ahmer's Dirac operators on quantum projective spaces \cite{Krahmer:Dirac, DanDab}, see also \cite{Matassa:CPq}, although that requires some additional analysis which will be deferred to a separate article.

\subsection{Summary of results}

We conclude the introduction with an overview of the paper, focusing on the comparison with the untwisted theory.   To get started, we must extend the twisting $\theta$ from $\sA$ to an algebra of generalized differential operators $\sD$.  As above, \emph{we only require that the twisting be a linear isomorphism, not an algebra automorphism}.  

In Section \ref{sec:equivalence}, following \cite{Uuye:PsiDOs}, we state equivalences between the existence  of various auxiliary structures: differential operators, pseudodifferential operators and pseudodifferential operators of order at most zero.  It should be emphasized, however, that in practice the key point is the passage from differential to pseudodifferential operators (\emph{i.e.}, algebra to analysis), which is detailed in Section \ref{sec:DO_to_PsiDO}.  

It is here that we impose a simplifying assumption on the twisting---namely ``diagonalizability''  (Definition \ref{def:diagonalizable}).  This assumption is very natural when one is motivated, as we are, by quantum groups.  It does not cover the conformally deformed spectral triples studied in \cite{ConMos:twisted}, although those examples are sufficiently close to the classical situation that the classical pseudodifferential calculus can be used.  See also \cite{PonWan:I}.

In Section \ref{sec:DO_to_PsiDO} we obtain an explicit asymptotic expansion for a product of pseudodifferential operators.  The formula needs
a quantum generalization of the binomial coefficients, which we develop in Appendix \ref{sec:mu-Cauchy}.  As in the untwisted case, this yields a residue trace under the hypothesis of simple dimension spectrum---see Section \ref{sec:zeta-functions}.  This may be a twisted trace, although \emph{a priori} the twisting here has no relation to the twisting of the spectral triple.

To indicate how this will apply in practice, we sketch in Section \ref{sec:Hopf} a framework from quantum groups which, in the presence of elliptic estimates, can lead to regular twisted spectral triples via our twisted pseudodifferential calculus.  The case of the Podle\'s sphere is discussed in some detail, \emph{cf.} \cite{NesTus:local_index_formula}.  Further examples will be discussed in a future paper.


\section{Sobolev theory}

Throughout this paper, we will fix a strictly positive 
unbounded operator $\Delta$ on a Hilbert space $H$, which we will think of as an \emph{abstract Laplace operator}. We will also fix an integer $r\geq2$, which is nominally the "order" of $\Delta$.  Typically, $r=2$.

Such an operator gives rise to an abstract Sobolev theory.  This is well summarized in the open sections of \cite{Uuye:PsiDOs}, to which we refer for details.  Let us quickly review the main points.

\subsection{Sobolev spaces}

Let $H^\infty = H^\infty(\Delta)$ denote the common domain of all powers of $\Delta$:
\[
  H^\infty \defeq \bigcap_{n=0}^\infty \dom(\Delta^n).
\]
The $s$th-Sobolev space $H^s = H^s(\Delta)$ is the completion of $H^\infty$ with respect to the inner product
\[
  \ip{\eta,\xi}_s \defeq \ip{ \Delta^\frac{s}{r}\eta, \Delta^\frac{s}{r}\xi } .
\]

\subsection{Operators of finite analytic order}
A linear operator on $T:H^\infty \to H^\infty$ is said to have \emph{analytic order (at most) $t\in\RR$} if, for every $s\in\RR$, it extends to bounded operator $T:H^s \to H^{s-t}$.  We write $\Op^t$ for the set of operators of analytic order at most $t$, and 
\[
  \Op \defeq \bigcup_{t\in\RR} \Op^t, \qquad \Op^{-\infty} \defeq \bigcap_{t\in\RR} \Op^t.
\] 
Then $\Op$ is an $\RR$-filtered algebra.  Also $\Op^0$ is an algebra of bounded operators on $\sH$, in which $\Op^{-t}$ is a two-sided ideal for all $t\in(0,\infty]$.  

For any $z\in\CC$, $\Delta^{\frac{z}{r}} \in \Op^{\Re(z)}$ is an isometric isomorphism from $H^s$ to $H^{s-\Re(z)}$ for every $s\in\RR$.   In particular it belongs to $\Op^{\Re(z)}$. 
It follows that $\Delta^{\frac{z}{r}}\Op^t =  \Op^{t+\Re(z)}$ and $\Op^t\Delta^{\frac{z}{r}} =  \Op^{t+\Re(z)}$ for all $z\in\CC$, $t\in\RR$.

\subsection{The $\Op$-topology}
\label{sec:topology}

Interpolation methods show that an operator $T:H^\infty \to H^\infty$ belongs to $\Op^t$ if and only if it extends continuously to a map $H^n \to H^{n-t}$ for all integers $n\in\ZZ$.  The family of operator norms $\left(\|\slot \|_{H^n \to H^{n+t}}\right)_{n\in\ZZ}$ therefore makes $\Op^t$ into a Fr\'echet space, and $\Op^0$ into a Fr\'echet algebra.  Note that these seminorms can also be written as 
\begin{equation}
 \label{eq:Frechet_norms}
 \|T\|_{H^n \to H^{n-t}} = \| \Delta^{\frac{n-t}{r}} T \Delta^{-\frac{n}{r}} \|_{\LL(H)}.
\end{equation}

Although this is the correct topological structure to place on the algebras associated to regular twisted spectral triples, we will rarely have need for it; see Remark \ref{rmk:ignore_topology}.

\section{Generalized differential and pseudodifferential operators}
\label{sec:definitions}

\subsection{Differential operators}
\label{sec:DOs}

Let $\sD$ be an $\NN$-filtered algebra, represented as linear operators on $H^\infty$.   The \emph{algebraic order} of $X\in\sD$ is $\algord(X) = \inf\{k\in\NN \st X \in \sD^k\}$.


As above, we fix an abstract Laplace operator $\Delta$ of degree $r$.  We will denote by $\nabla$ the twisted commutator
\[
  \nabla(X) \defeq [\Delta , X]_{\theta^r} = \Delta X - \theta^{r}(X) \Delta, \qquad  X\in\End(H^\infty).
\]

\begin{definition}
 \label{def:DOs}
 A \emph{twisted algebra of abstract differential operators} (abbreviated to \emph{twisted algebra of DOs}) associated to $\Delta$ is an $\NN$-filtered algebra $\sD$ of operators on $H^\infty$ equipped with a linear filtration-preserving automorphism $\theta$, such that:
 \begin{enumerate}
  \item The twisted commutator $\nabla = [\Delta, \slot]_{\theta^r}$ maps $\sD^m$ to $\sD^{m+r-1}$
  \item \emph{Elliptic estimate}: For any $X\in\sD^m$, there is $C>0$ such that for any $v\in H^\infty$, $\|Xv\|_H \leq C \|\Delta^{\frac{m}{r}} v\|_H$.
 \end{enumerate}
\end{definition}

\begin{remark}
\label{rmk:symbol_aut}
Note that we do not require the twisting $\theta$ to be an algebra automorphism on $\sD$.  Nevertheless, for any $X\in\sD^m$, $Y\in\sD^n$ we have
\[
 (\theta^r(XY)-\theta^r(X)\theta^r(Y))\Delta = \nabla(X) Y + X \nabla(Y) - \nabla(XY) 
   \in \sD^{m+n+r-1}.
\]
This shows that, at least under some mild assumptions on $\Delta$,
(\emph{e.g.}, that it is a multiplier of $\sD$ which is injective on the associated graded algebra), 
the twisting $\theta^r$ does define an algebra automorphism on the associated graded algebra of $\sD$---which one might reasonably call the algebra of \emph{principal symbols}.

\end{remark}

The basic estimate is equivalent to a compatibility between the algebraic and analytic order of differential operators.  This is made precise by the following Lemma, which has essentially the same proof
as its untwisted analogue in \cite{Higson:local_index_formula} or \cite{Uuye:PsiDOs}.  

\begin{lemma}
 \label{lem:analytic_order}
 Let $\sD$ be an $\NN$-filtered algebra of linear operators on $H^\infty$ such that $[\Delta,\sD^m]_{\theta^r} \in \sD^{m+r-1}$ for all $m \in \NN$.
 Then the elliptic estimate is satisfied for every $X \in \sD^m$ if and only if $\sD^m \subset \Op^m$ for all $m\in\NN$.
 
\end{lemma}

\begin{remark}
 Definition \ref{def:DOs} does not force an \emph{equality} of algebraic and analytic order---\emph{i.e.}, we may not have $\sD^m = \sD \cap \Op^m$ for all $m$.   For an obviously artificial example, take the classical Laplace operator $\Delta$ on $L^2(M)$ for a smooth Riemannian manifold $M$, and let $\sD=\DO(M)$ be the algebra of differential operators but with the shifted filtration:
\[
 \sD^m = \begin{cases}
        \CC, & m=0\\
        \DO^{m-1}(M), & m\geq 1.
       \end{cases}
\]
This satisfies the axioms of a twisted algebra of $\DO$s with trivial twisting.
\end{remark}


\subsection{Pseudodifferential operators}
\label{sec:PsiDOs}

To define pseudodifferential operators, we must incorporate complex powers of the Laplacian.  
But we begin with the appropriate notion of twisting in this context.  

Let $\Psi$ be an $\RR$-filtered subalgebra of $\Op$.  
In this context, a \emph{twisting} of $\Psi$ will be given by a complex one-parameter family of algebra automorphisms $(\Theta^z)_{z\in\CC}$ which preserves the filtration.   We will write $\Theta = \Theta^1$.  

\begin{remark}
Unlike for the algebras of differential operators above, here we will have the liberty to demand that $\Theta$ be an automorphism of the algebra $\Psi$, not just of the associated graded algebra, which in any case is problematic to define for an $\RR$-filtered algebra without some additional structure.
\end{remark}

\begin{definition}
 \label{def:GPsiDOs}
 A \emph{twisted algebra of abstract pseudodifferential operators} (abbreviated to \emph{twisted algebra of $\Psi$DOs}) is a subalgebra $\Psi \subseteq \Op$ equipped with a one-parameter family of algebra automorphisms $(\Theta^z)_{z\in\CC}$ such that
 \begin{enumerate}
  \item  $\Delta^{z}\Psi \subseteq \Psi$ and $\Psi\Delta^{z} \subseteq \Psi$ for all $z\in\CC$,
  \item $[\Delta^{\frac{z}{r}},\Psi^t]_{\Theta^z} \subseteq \Psi^{\Re(z)+t-1}$ for all $z\in\CC$, $t\in\RR$,
  \item $\Psi^0 \subseteq \Op^0$.
 \end{enumerate}

 Two twistings $\Theta^\bullet$ and ${\Theta'}^\bullet$ on $\Psi$ will be called \emph{equivalent} if for all $T\in\Psi^t$ and all $z\in\CC$, $\Theta^z(T)-{\Theta'}^z(T) \in \Psi^{t-1}$.
 
 We will say $\Delta$ is \emph{$\Theta$-central} if $[\Delta^{\frac{z}{r}},\Psi]_{\Theta^z} = 0$ for all $z\in\CC$.
\end{definition}

\begin{lemma}
 \label{lem:central}
 Let $\Psi$ and $\Theta^\bullet$ be as above.  There is an equivalent twisting ${\Theta'}^\bullet$ on $\Psi$ such that $\Delta$ is $\Theta'$-central.
\end{lemma}

\begin{proof}
 One can take ${\Theta'}^z(T) = \Delta^{\frac{z}{r}} T \Delta^{-\frac{z}{r}}$. The identity
 \[
  \Theta^\prime(T) - \Theta(T) = [\Delta^\frac{1}{r}, T]_\Theta \Delta^{-\frac{1}{r}}
 \]
 shows that the two twistings are equivalent.
\end{proof}

From the invertibility of $\Delta$ on $H^\infty$, Condition (1) of Definition \ref{def:GPsiDOs} implies that $\Delta^\frac{z}{r} \Psi^t = \Psi^{\Re(z)+t}= \Psi^t \Delta^\frac{z}{r}$ for all $z\in\CC$, $t\in\RR$, and with Condition (3) we also get $\Psi^z \subseteq \Op^{\Re(z)}$.  In particular, the algebra $\Psi$ is completely determined by its subalgebra $\Psi^0$ of elements of order at most zero.  This motivates the next definition.

\subsection{Pseudodifferential operators of order at most $0$}
\label{sec:PsiDOos}

\begin{definition}
 A \emph{twisted algebra of abstract pseudodifferential operators of order at most zero} (abbreviated to \emph{twisted algebra of $\Psi$DO${}_0$s}) is an algebra $\sB$ of bounded operators on $H^\infty$, equipped with a linear automorphism $\theta$ such that
  $\sB$ is closed under the twisted derivation $\delta_\theta := [\Delta^\frac{1}{r},\slot]_\theta$.
\end{definition}

This is the structure which is closest to regularity for twisted spectral triples; see Section \ref{sec:regular}.

Note that, for all $b\in\sB$, we have
\begin{align}
\label{eq:B_conjugates}
  \Delta^\frac{1}{r} b \Delta^{-\frac{1}{r}} &= \theta(b) + \delta_\theta(b) \Delta^{-\frac{1}{r}} \in \sB + \sB\Delta^{-\frac{1}{r}}, \\
\label{eq:B_conjugates2}
    \Delta^{-\frac{1}{r}} b \Delta^{\frac{1}{r}} &= \theta^{-1}(b) + \Delta^{-\frac{1}{r}} \delta_\theta(\theta^{-1}(b))  \in \sB + \Delta^{-\frac{1}{r}}\sB,
\end{align}
Induction on $n$ shows that $\Delta^{\frac{n}{r}} b \Delta^{-\frac{n}{r}}$ is bounded for all $n\in\ZZ$.  This proves the following fact.

\begin{lemma}
\label{lem:B_in_Op}
If $\sB$ is a twisted algebra of $\PsiDO_0$s then $\sB \subseteq \Op^0$. 
\end{lemma}

\begin{remark}
\label{rmk:ignore_topology}
If one wanted to topologize $\sA$ or $\sB$, the $\Op^0$-topology of Section \ref{sec:topology} would be the appropriate one.  Given that, we should insist that the twisting $\theta$ be $\Op^0$-continuous.  In fact, for our main theorem, we will work in a much more algebraic context---namely, diagonalizable twistings (Definition \ref{def:diagonalizable})---which does not require $\Op^0$-continuity.  For these reasons, we will usually sweep the topology under the rug.
\end{remark}


\section{Equivalence of definitions}
\label{sec:equivalence}

This section is dedicated to the equivalence of the various notions above.     As usual, we fix an abstract Laplace operator $\Delta$ of order $r$.

\begin{definition}
 \label{def:compatible}
 Let $\sA$ be an algebra of operators on $H^\infty$ with a twisting $\theta$. 
 \begin{enumerate}
  \item A twisted algebra of DOs $\sD$ is \emph{compatible with $\sA$} if $\sA\subseteq \sD^0$ and its twisting $\theta$ extends that of $\sA$.
  \item A twisted algebra of $\PsiDO_0$s $\sB$ is \emph{compatible with $\sA$} if $\sA\subseteq\sB$ and its twisting $\theta$ extends that of $\sA$.
  \item A twisted algebra of $\PsiDO$s $\Psi$ is \emph{compatible with $\sA$} if $\sA\subseteq\Psi^0$ and its twisting $\Theta$ satisfies $\Theta(a) - \theta(a) \in \Psi^{-1}$ for all $a\in\sA$.
 \end{enumerate}
\end{definition}

One complication which arises in the twisted case is that, in passing from differential to pseudodifferential operators, we need to extend the twisting $\theta$ of $\sD$ to a complex one-parameter family of automorphisms $\Theta$ of $\Psi$.  Various conditions can be imposed to ensure this.  In this article, motivated by the examples arising in quantum groups, we work with a very algebraic condition on $\theta$.

\begin{definition}
\label{def:diagonalizable}
 We will say a linear map $\theta$ on a vector space $\sD$ is \emph{diagonalizable} if $\sD$ is the algebraic direct sum of the eigenspaces of $\theta$.
\end{definition}

\begin{theorem}
 Fix an abstract Laplace operator $\Delta$ of order $r$ and let $\sA$ be an algebra of linear operators on $H^\infty$ equipped with an algebra automorphism $\theta$.  
 \begin{enumerate}
  \item If $\sA$ admits a compatible twisted algebra of $\PsiDO$s, then it admits a compatible twisted algebra of $\PsiDOo$s.
  \item If $\sA$ admits a compatible twisted algebra of $\PsiDOo$s, then it admits a compatible twisted algebra of $\DO$s.
  \item If $\sA$ admits a compatible twisted algebra of $\DO$s, such that the twisting on this algebra is diagonalizable with positive spectrum, then it admits a compatible twisted algebra of $\PsiDO$s.
 \end{enumerate}
\end{theorem}

The first two statements are relatively straightforward, and we shall deal with them rapidly in the following two subsections.  The most profound of the three claims---and also the most useful in practice---is (3), which merits its own section (Section \ref{sec:DO_to_PsiDO}).

\subsection{From pseudodifferential operators to pseudodifferential operators of order at most $0$}

Let $\Psi$ be a twisted algebra of $\PsiDO$s compatible with $\sA$. Denote by $\theta$ the twisting of $\sA$ and by $\Theta$ the twisting of $\Psi$.   Consider the linear map $\kappa \defeq \theta-\Theta|_\sA:\sA \to \Psi^{-1}$.  If we make an arbitrary linear extension of this to a map $\kappa:\Psi^0 \to \Psi^{-1}$, then defining $\theta^\prime \defeq \Theta +\kappa: \Psi^0 \to \Psi^0$ gives an extension of $\theta:\sA\to\sA$.  Moreover, for any $b\in\Psi^0$ we have
\[
  \delta_{\theta^\prime}(b) = [\Delta^\frac{1}{r} ,b]_{\Theta}  - \kappa(b) \Delta^\frac{1}{r} 
   \in \Psi^0,
\]
so that $\Psi^0$ is invariant under $\delta_{\theta^\prime}$.  Hence $\Psi^0$ is a twisted algebra of $\PsiDOo$s compatible with $\sA$.

\subsection{From pseudodifferential operators of order at most $0$ to differential operators}

Let $\sB$ be a twisted algebra of $\PsiDOo$s. We say that  $\sB^\prime$ is an \emph{extension} of $\sB$ if $\sB \subseteq \sB^\prime$ and there is a linear extension of $\theta$ which makes $\sB^\prime$ into a twisted algebra of $\PsiDOo$s. If $\sB$ is compatible with an algebra $\sA$ then the same is true for $\sB^\prime$.

\begin{lemma}
\label{lem:extension-pdo0}
Let $\sB$ be a twisted algebra of $\PsiDOo$s. Then there is an extension $\sB^\prime$ of $\sB$ which contains $\Delta^{-\frac{1}{r}}$.  Then $\Theta \defeq \Delta^\frac{1}{r}\slot\Delta^{-\frac{1}{r}}$ is an automorphism of $\sB'$.
\end{lemma}

\begin{proof}
We define $\sB^\prime$ as the algebra generated by $\sB$ and $\Delta^{-\frac{1}{r}}$.
Then $\sB^\prime \subseteq \mathrm{Op}^0$ since $\Delta^{-\frac{1}{r}} \in \mathrm{Op}^0$.
Equations \eqref{eq:B_conjugates} and \eqref{eq:B_conjugates2} show that $\Theta$ is a well-defined automorphism of $\sB^\prime$.  Now we consider the difference between $\theta$ and $\Theta|_\sB$. For all $b\in \sB$ we have
\[
 \theta(b) - \Theta(b) = (\theta(b)\Delta^\frac{1}{r} - \Delta^\frac{1}{r} b) \Delta^{-\frac{1}{r}} 
 = -\delta_\theta(b) \Delta^{-\frac{1}{r}} 
   \in \sB \Delta^{-\frac{1}{r}}.
\]
We let $\kappa : \sB^\prime \to \sB^\prime \Delta^{-\frac{1}{r}}$ be an arbitrary, not necessarily continuous, linear extension of the map $\theta - \Theta|_\sB : \sB \to \sB \Delta^{-\frac{1}{r}}$. Then defining
\[
\theta^\prime = \Theta + \kappa : \sB^\prime \to \sB^\prime
\]
gives a linear extension of $\theta : \sB \to \sB$.
Finally to show that $\sB^\prime$ is closed under the twisted derivation $\delta_{\theta^\prime} = [\Delta^\frac{1}{r}, \slot]_{\theta^\prime}$ we write
\[
 [\Delta^\frac{1}{r}, x]_{\theta^\prime} = [\Delta^\frac{1}{r}, x]_\Theta + (\Theta(x) - \theta^\prime(x)) \Delta^{\frac{1}{r}}
 = - \kappa(x) \Delta^\frac{1}{r} \in \sB^\prime
\]
for $x\in\sB'$, where we have used the fact that $\kappa(x) \in  \sB^\prime \Delta^{-\frac{1}{r}}$.
\end{proof}

In view of the previous lemma, we will assume in the following that $\sB$ is a twisted algebra of $\PsiDOo$s compatible with $\sA$ such that $\Delta^{-\frac{1}{r}} \in \sB$.

\begin{proposition}
Let $\sD = \bigcup_{m \in \mathbb{N}} \sD^m$ where
\[
\sD^m = \sum_{k = 0}^m \sB\Delta^{\frac{k}{r}}.
\]
Then there is an extension of $\theta : \sB \to \sB$ to a linear automorphism of $\sD$ making it into a twisted algebra of $\DO$s compatible with $\sA$.
\end{proposition}

\begin{proof}
It follows from Lemma \ref{lem:B_in_Op} that $\sD^m \subseteq \Op^m$ for all $m \in \mathbb{N}$.  Therefore, by Lemma \ref{lem:analytic_order}, it suffices to show that $[\Delta,\sD^m]_{\theta^\prime} \subseteq \sD^{m + r - 1}$, where $\theta^\prime$ is an appropriate linear extension of $\theta$ to $\sD$.

Consider the automorphism $\Theta: \sD \to \sD$ defined by $\Theta(X) = \Delta^\frac{1}{r} X \Delta^{-\frac{1}{r}}$. It is well defined since $\sB$ is stable under conjugation by $\Delta^\frac{1}{r}$. Moreover it preserves the filtration of $\sD$. Indeed for $b \Delta^{\frac{m}{r}} \in \sD^m$ we have
\[
\Theta(b \Delta^{\frac{m}{r}}) = \delta_\theta(b) \Delta^{\frac{m - 1}{r}} + \theta(b) \Delta^{\frac{m}{r}} \in \sD^m.
\]
Proceeding as in Lemma \ref{lem:extension-pdo0}, we let $\kappa : \sD \to \sD \Delta^{-\frac{1}{r}}$ be an arbitrary linear extension of the map $\theta - \Theta|_\sB : \sB \to \sB \Delta^{-\frac{1}{r}}$ such that $\kappa:\sD^m \to \sD^{m}\Delta^{-\frac{1}{r}}$ for all $m$. Then we define the linear map
\[
 \theta^\prime = \Theta + \kappa : \sD \to \sD,
\]
which gives a filtration-preserving extension of $\theta$. Now we look at the twisted commutator condition. First, for all $X \in \sD^m$ we have
\[
 [\Delta^\frac{1}{r} , X]_{\theta^\prime} = [\Delta^\frac{1}{r}, X]_\Theta + (\Theta(X) - \theta^\prime(X)) \Delta^{\frac{1}{r}}
 = - \kappa(X) \Delta^\frac{1}{r} \in \sD^m.
\]
Next for all $n \in \mathbb{N}$ we have the algebraic identity
\[
[\Delta^\frac{n}{r}, X]_{\theta^{\prime n}} = \Delta^\frac{1}{r} [\Delta^\frac{n - 1}{r}, X]_{\theta^{\prime n - 1}} + [\Delta^\frac{1}{r}, \theta^{\prime n - 1}(X)]_{\theta^\prime} \Delta^\frac{n - 1}{r}.
\]
Since $[\Delta^\frac{1}{r} , X]_{\theta'} \in \sD^m$, induction in $n$ shows that $[\Delta^\frac{n}{r}, X]_{\theta^{\prime n}} \in \sD^{m + n - 1}$ for all $n$. We find in particular that $[\Delta, X]_{\theta^{\prime r}} \in \sD^{m + r - 1}$, which concludes the proof.
\end{proof}


\section{From differential operators to pseudodifferential operators}
\label{sec:DO_to_PsiDO}

As already mentioned, the key point in passing to pseudodifferential operators is to introduce complex powers of the Laplace operator.  This is achieved as follows.

\begin{definition}
\label{def:PsiDO}
 Let $\sD$ be a twisted algebra of $\DO$s.  A linear operator $P$ on $H^\infty$ is called a \emph{basic} (or \emph{step 1}) \emph{pseudodifferential operator} of order at most $t\in\RR$ if, for any $l\in\RR$ there exists a decomposition of the form
 \begin{equation}
 \label{eq:PsiDO}
  P = X \Delta^{\frac{z-m}{r}} + Q,
 \end{equation}
 where
 \begin{itemize}
  \item $X \in \sD^m$, for some $m\in\NN$,
  \item $\Re(z) \leq t$, and
  \item $Q \in \Op^l$.
 \end{itemize}
 Then $\Psi^t$ is defined to be the space of finite sums of basic pseudodifferential operators of order at most $t$.
\end{definition}

The main technical point is to prove that $\Psi$ is an algebra.  From the definition, one sees that the key issue is to commute a complex power of $\Delta$ past $X'\in\sD$.  
For this, the main tool is the Cauchy integral formula:
\begin{equation}
   \Delta^{z} = \frac{1}{2\pi i} \int_\Gamma \lambda^z (\lambda-\Delta)^{-1} \, d\lambda
\end{equation}
where $\Gamma$ is a vertical contour which separates the spectrum of $\Delta$ from $0$, and $\Re(z)<0$.  We must therefore take a short digression through the analysis of such formulas.


\subsection{Resolvent identities}

We begin with a twisted algebra of $\DO$s $\sD$ compatible with $\sA$.  Since we are assuming that the twisting $\theta$ is diagonalizable, with positive spectrum, we can unambiguously define the complex powers $\theta^z$ with $z\in\CC$.  
We will refer to the eigenvalues of $\theta^r$ as \emph{weights}.   

Recall that we write $\nabla = [\Delta,\slot]_{\theta^r}$ for the twisted commutators with $\Delta$.
We begin by deriving some identities involving the resolvent of $\Delta$.  If $\mu\in\RR_+$, we will put 
\begin{equation*}
R(\mu) = (\lambda - \mu \Delta)^{-1}.
\end{equation*}
More generally for a multi-index $\bfmu = (\mu_0, \mu_1, \cdots, \mu_k)$ we set $R(\bfmu) = \prod_{i = 0}^k R(\mu_i)$.

\begin{lemma}
Let $X\in\sD$ be homogeneous of weight $\mu$, that is $\theta^r(X) = \mu X$. Then
\[
R(1) X = X R(\mu) + R(1) \nabla(X) R(\mu).
\]
\end{lemma}

\begin{proof}
This follows by multiplying the identity
 \[
  \nabla(X) = \Delta X - \mu X \Delta = X (\lambda - \mu \Delta) - (\lambda - \Delta) X.
 \]
 on the left and the right by $R(1)$ and $R(\mu)$, respectively.
\end{proof}

For successive iterations of this formula, we will need to decompose $\nabla(X)$ and its iterates into homogeneous components.  Let $W\subset\RR_+$ be the set of weights of $\theta^r$.  If $X\in\sD$ and $\mu\in W$, let us write $X^\mu$ for its component of weight $\mu$.  Under our hypothesis of diagonalizability, we have $X= \sum_\mu X^\mu$ and the sum always contains only finitely many nonzero terms.  

\begin{definition}
 Let us write $W(k) = W^{k+1}$ for the set of $(k+1)$-tuples of weights.  For $\bfmu = (\mu_0,\ldots,\mu_k) \in W(k)$, we define $\nabla^{\bfmu}(X)$ iteratively by setting $\nabla^{\bfmu}(X) = X^{\mu_0}$ when $k = 0$ and 
\[\nabla^{\bfmu}(P) = (\nabla(\nabla^{\bfmu^\prime}(P)))^{\mu_k}, \quad \text{when }\bfmu^\prime = (\mu_0, \mu_1, \cdots, \mu_{k - 1}).
\]
\end{definition}


\begin{proposition}
\label{prop:expansion}
Let $X \in \mathcal{D}$.  For any $n \in \mathbb{N}$ we have the expansion
\begin{equation}
\label{eq:expansion}
R(1) X = \sum_{k = 0}^n \sum_{\bfmu \in W(k)} \nabla^{\bfmu}(X) R(\bfmu) + R(1) \sum_{\bfmu \in W(n)} \nabla (\nabla^{\bfmu}(X)) R(\bfmu).
\end{equation}
\end{proposition}

\begin{proof}
For $n = 0$ the identity is
\[
R(1) X = \sum_{\mu \in W} X^\mu R(\mu) + R(1) \sum_{\mu \in W} \nabla(X^\mu) R(\mu),
\]
which holds by the previous lemma.
We proceed by induction, assuming Equation \eqref{eq:expansion} holds for some $n$.
Let $\bfmu = (\mu_0, \mu_1, \cdots, \mu_n) \in W(n)$. We have
\[
\nabla(\nabla^{\bfmu}(X)) = \sum_{\mu_{n + 1} \in W} \nabla^{\bfmu^\prime}(X),
\]
where we have defined $\bfmu^\prime = (\mu_0, \mu_1, \cdots, \mu_n, \mu_{n + 1}) \in W(n + 1)$. The term $\nabla^{\bfmu^\prime}(X)$ has weight $\mu_{n + 1}$. Then we can use the lemma to rewrite
\[
R(1) \nabla^{\bfmu^\prime}(X) = \nabla^{\bfmu^\prime}(X) R(\mu_{n + 1}) + R(1) \nabla(\nabla^{\bfmu^\prime}(X)) R(\mu_{n + 1}).
\]
Applying this to the last term of Equation \eqref{eq:expansion}, we get
\[
\begin{split}
R(1) X & = \sum_{k = 0}^n \sum_{\bfmu \in W(k)} \nabla^{\bfmu}(X) R(\bfmu) + \sum_{\bfmu \in W(n)} \sum_{\mu_{n + 1} \in W} \nabla^{\bfmu^\prime}(X) R(\mu_{n + 1}) R(\bfmu) \\
& + R(1) \sum_{\bfmu \in W(n)} \sum_{\mu_{n + 1} \in W} \nabla(\nabla^{\bfmu^\prime}(X)) R(\mu_{n + 1}) R(\bfmu).
\end{split}
\]
Since $R(\mu_{n + 1}) R(\bfmu) = R(\bfmu^\prime)$, we obtain the result.
\end{proof}


\subsection{Commutators with complex powers of the Laplacian}
\label{sec:Psi}

The next step is to use the resolvent expansion of Equation \eqref{eq:expansion} and the Cauchy Integral Formula to obtain an expansion of $\Delta^z X$ with complex powers of $\Delta$ on the right.  

The resulting formulas involve certain generalized binomial coefficients $\binom{z}{n}_{\!\bfmu}$, of which both the standard and $q$-binomial coefficients are special cases.   We will leave these constants as a black box for the moment.  The details are given in Appendix \ref{sec:mu-binomial}, where we also prove the following quantum analogue of Cauchy's Integral Formula.

\begin{lemma}
\label{lem:mu-Cauchy-Delta}
  Let $\bfmu=(\mu_0,\ldots,\mu_n) \in W(n)$.  For any $z\in\CC$ with $\Re(z)<0$ we have
  \[
    \frac{1}{2\pi i} \int_\Gamma \lambda^z R(\bfmu) \, d\lambda = \binom{z}{n}_{\!\!\bfmu}\Delta^{z-n},
  \]
  where $\Gamma$ is a vertical contour separating the spectrum of $\mu_i\Delta$ from $0$ for every $i$, and where $\binom{z}{n}_{\!\bfmu}$ is the generalized binomial coefficient of Section \ref{sec:mu-binomial}.
\end{lemma}

\begin{proposition}
\label{prop:Delta_expansion}
Let $z \in \mathbb{C}$ and $Y \in \sD^m$. Then for any $n \in \mathbb{N}$ we have
\[
\Delta^z Y =  \sum_{k = 0}^n \sum_{\bfmu\in W(k)} \binom{z}{k}_{\!\!\bfmu} \nabla^\bfmu(Y) \Delta^{z - k} \; + \; Q,
\]
for some $Q \in \Op^{\Re(z) - n - 1}$.   The highest order term (\emph{i.e.} $k=0$) is equal to $\theta^{rz}(Y) \Delta^z$.
\end{proposition}

\begin{proof}
To begin with, suppose $\Re(z)<0$.  Applying the Cauchy Integral formula, and using Proposition \ref{prop:expansion} and Lemma \ref{lem:mu-Cauchy-Delta}, we get
\begin{align}
 \Delta^z Y 
  &= \int_\Gamma \lambda^z R(1) Y \, dz  \nonumber \\
  &= \sum_{k=0}^n \sum_{\bfmu\in W(k)} \nabla^\bfmu(Y) \int_\Gamma \lambda^z R(\bfmu) \, dz 
   \nonumber \\
  & \qquad\qquad       + \!\! \sum_{\bfmu\in W(n)} \int_\Gamma \lambda^z R(1) \nabla(\nabla^\bfmu(Y)) R(\bfmu) \, dz
   \nonumber \\
  &= \sum_{k=0}^n \sum_{\bfmu\in W(k)} \binom{z}{k}_{\!\!\bfmu} \nabla^\bfmu(Y) \Delta^{z-k} 
   \nonumber \\
  & \qquad \qquad      + \!\! \sum_{\bfmu\in W(n)} \int_\Gamma \lambda^z R(1) \nabla(\nabla^\bfmu(Y)) R(\bfmu) \, dz  ,
    \label{eq:integral}
\end{align}
where $\Gamma$ is a vertical contour separating $0$ from the spectrum of $\mu_i\Delta$ for every $\mu_i$ appearing (non-trivially) in the formula.  By Proposition \ref{prop:mu-binomial}, the $k=0$ term in the sum equals
\[
  \sum_{\mu\in W} \mu^z Y^\mu \Delta^z = \theta^{rz}(Y)\Delta^z,
\]
as claimed.

For the remainder term, note that $\nabla(\nabla^\bfmu(Y)) \in \sD^{m+(n+1)(r-1)}$, so that $\nabla(\nabla^\bfmu(Y))R(\bfmu)$ belongs to $\Op^{m-n-1}$, and for any $s\in\RR$ its norm as an operator from $H^{s+m-n-1}$  to $H^s$ is uniformly bounded in $\lambda$.  Therefore, the integrals in the last line of \eqref{eq:integral} all converge uniformly in $\Op^{m-n-1}$.

This completes the proof when $\Re(z)<0$.  If the result holds for $z\in\CC$, then for $z+1$ we obtain
\begin{align*}
 \Delta^{z+1} Y &= \Delta^{z} \theta^r(Y) \Delta + \Delta^{z} \nabla(Y) \\
    &= \sum_{k=0}^n \sum_{\bfmu\in W(k)} \binom{z}{k}_{\!\!\bfmu} \nabla^\bfmu(\theta^r(Y)) \Delta^{z-k+1} \\
    & \qquad + \sum_{k=0}^n \sum_{\bfmu\in W(k)} \binom{z}{k}_{\!\!\bfmu} \nabla^\bfmu(\nabla(Y)) \Delta^{z-k}
     \;+\; Q \\
    &= \sum_{k=0}^n \sum_{\bfmu\in W(k)} \mu_0 \binom{z}{k}_{\!\!\bfmu}  \nabla^\bfmu(Y) \Delta^{z-k+1} \\
    & \qquad + \sum_{k=1}^{n+1} \sum_{\bfmu\in W(k)} \binom{z}{k-1}_{\!\!\check{\bfmu}} \nabla^\bfmu(Y) \Delta^{z-(k-1)}
     \;+\; Q,
\end{align*}
where $\check{\bfmu} = (\mu_1,\ldots,\mu_k)$ is the $k$-tuple obtained by removing $\mu_0$ and the remainder $Q$ is in $\Op^{\Re(z)-n}$.  By Pascal's Identity (Proposition \ref{prop:mu-binomial} (2)) we obtain
\[
  \Delta^{z+1} Y = \sum_{k = 0}^n \sum_{\bfmu\in W(k)} \binom{z+1}{k}_{\!\!\bfmu} \nabla^\bfmu(Y) \Delta^{z+1 - k} \;+\; Q',
\]
for some $Q' \in \Op^{\Re(z)-n}$.
An induction finishes the proof.
\end{proof}


\begin{definition}
 Let $P$ and $P_n$ ($n\in\NN$) be operators in $\Op$.  We say that $P$ admits the \emph{asymptotic expansion} $P\sim \sum_n P_n$ if, for every $l\in\RR$, there is $N\in\NN$ such that for all $n\geq\NN$,
 \[
  P - \sum_{k=0}^n P_k \in \Op^l.
 \]
\end{definition}

With this terminology, Proposition \ref{prop:Delta_expansion} can be rephrased as saying that $\Delta^zY$ admits the asymptotic expansion
\begin{equation}
\label{eq:Delta_asymptotic}
 \Delta^z Y \sim \theta^{r z}(Y) \Delta^z +  \sum_{k = 1}^\infty \sum_{\bfmu\in W(k)} \binom{z}{k}_{\!\!\bfmu} \nabla^\bfmu(Y) \Delta^{z - k}.
\end{equation}

\begin{lemma}
 The space $\Psi$ is an $\mathbb{R}$-filtered subalgebra of $\Op$ with $\Psi^0 \subseteq \Op^0$.
\end{lemma}

\begin{proof}
It is immediate from the definition that $\Psi^t \subseteq \Op^t$ for all $t\in\RR$.
We need to show that, given two basic pseudodifferential operators $P \in \Psi^t$ and $P^\prime \in \Psi^{t^\prime}$, their product $P P^\prime$ belongs to $\Psi^{t + t^\prime}$.   For any $\ell\in\RR$, we can write $P = X_m \Delta^{(z - m)/r} + Q$ and $P' = X'_{m'} \Delta^{(z'-m')/r}+Q'$, with $Q,Q'\in\Op^\ell$, as in Definition \ref{def:PsiDO}. The product takes the form
\begin{equation}
\label{eq:product}
P P^\prime = X_m \Delta^{(z - m)/r} X'_{m^\prime} \Delta^{(z^\prime - m^\prime)/r} + Q'',
\end{equation}
where
\begin{equation}
\label{eq:product_remainder}
 Q'' = X_m \Delta^{(z - m)/r} Q' + Q X_{m^\prime} \Delta^{(z^\prime - m^\prime)/r} + Q Q'.
\end{equation}
The three summands of the remainder term \eqref{eq:product_remainder} belong to $\Op^{t+\ell}$, $\Op^{t'+\ell}$ and $\Op^{2\ell}$, respectively, so by choosing $\ell$ sufficiently large and negative, we can ensure that the analytic order of $Q''$ is as large and negative as we want.

Having done this, the first term in \eqref{eq:product} admits an asymptotic expansion by Equation \eqref{eq:Delta_asymptotic}:
\[
 X_m \Delta^{(z - m)/r} X'_{m^\prime} \Delta^{(z^\prime - m^\prime)/r} 
  \sim \sum_{k=0}^\infty \sum_{\bfmu\in W(k)} {\binom{(z-m)/r}{k}_{\!\!\bfmu}} X_m \nabla^\bfmu(X'_{m'})
   \Delta^{\frac{z+z'-m-m'}{r} - k},
\]
where $X_m \nabla^\bfmu(X'_{m'}) \in \sD^{m+m'+k(r-1)}$. 
It is therefore a pseudodifferential operator of order at most $t+t'$.  This completes the proof.
\end{proof}

Finally, we need to equip $\Psi$ with a twisting---\emph{i.e.}, a one-parameter family of automorphisms $(\Theta^z)_{z\in\CC}$---making it into a twisted algebra of $\PsiDO$s.

\begin{proposition}
\label{prop:gdo-to-gpdo}
The one-parameter family of automorphisms $\Theta^z \defeq \Delta^{\frac{z}{r}} \slot \Delta^{-\frac{z}{r}}$ preserves the algebra $\Psi$, and make it into a twisted algebra of $\PsiDO$s compatible with $\sA$.
\end{proposition}

\begin{proof}
The asymptotic expansion \eqref{eq:Delta_asymptotic} shows that $\Psi$ is preserved by $\Theta^z$ for every $z\in\CC$.   It also shows that $\Delta^{\frac{z}{r}}\Psi^t \subseteq \Psi^{\Re(z)+t}$, and clearly $\Psi^t\Delta^{\frac{z}{r}}\subseteq \Psi^{\Re(z)+t}$.  Moreover, $\Delta$ is $\Theta$-central, meaning $[\Delta^\frac{z}{r},P]_{\Theta^{z}} = 0$ for all $P\in \Psi$.  Finally, if $a\in\sA$ then 
\[
  \Theta(a)  - \theta(a) =  (\Delta^\frac{1}{r} a - \theta(a) \Delta^{\frac{1}{r}})\Delta^{-\frac{1}{r}}.
\]
By Proposition \ref{prop:Delta_expansion}, $\Delta^\frac{1}{r} a - \theta(a) \Delta^\frac{1}{r} \in \Psi^{0}$, so $\Theta(a) - \theta(a) \in \Psi^{-1}$.  This completes the proof.
\end{proof}


\section{Regularity of spectral triples}
\label{sec:regular}

Let $(\sA,H,D)$ be a twisted spectral triple (see Definition \ref{def:twisted_spectral_triple}) with twisting $\theta$.  We consider the abstract Laplace operator
\[
 \Delta := D^2+1
\]
with order $r=2$.   It is immediate from the definitions that $(\sA,H,D)$ is regular if and only if it admits a twisted algebra of $\PsiDOo$s which is compatible with $\sA$ and also contains $[D,\sA]_{\theta}$ in order $0$.  We immediately obtain the following.

\begin{theorem}
 Let $(\sA,H,D)$ be a  twisted spectral triple and put $\Delta = D^2+1$.  
 \begin{enumerate}
  \item If $\sA$ admits a compatible twisted algebra of $\DO$s $\sD$ which contains $[D,\sA]_\theta$ in order $0$ and such that $\theta$ is diagonalizable on $\sD$, then $(\sA,H,D)$ is regular.  Moreover, $\sA$ admits a compatible twisted algebra of $\PsiDO$s.
  \item Conversely, if $(\sA,H,D)$ is regular, then it admits a compatible twisted algebra of $\DO$s which also contains $[D,\sA]_\theta$ in order $0$.
 \end{enumerate}
\end{theorem}


\section{Examples from quantum groups}
\label{sec:Hopf}

Let $\sU$ be a unital Hopf algebra and $\sA$ a left $\sU$-module algebra with action denoted by $\triangleright$.  We can then form the cross-product algebra $\sA \# \sU$, which is the algebra generated by $\sA$ and $\sU$ with relations $h a = (h_{(1)} \triangleright a) h_{(2)}$ for $a \in \sA$ and $h \in \sU$. 

\begin{example}
 Let $G$ be a compact Lie group.  Let $\sU=\sU(\lie{g})$ be the universal enveloping algebra of its Lie algebra and let $\sA = C^\infty(G)$.  Then $\sA\#\sU$ is the algebra of differential operators on $G$.
\end{example}

\begin{example}
 Let $K$ be a compact semisimple Lie group, and let $\lie{g}=\lie{k}_\CC$ be the complexified Lie algebra.
 Put $\sU = \sU_q(\lie{g})$ and let $\sA = \sO(K_q)$.  Then $\sA\#\sU$ is an analogue of the algebra of polynomial differential operators on the quantum group $K_q$.
\end{example}

This second example is inspirational, but is somewhat too complicated for the simple framework we will describe here.  Still, our framework does apply to certain quantum homogeneous spaces, as we will indicate shortly. 

\medskip
Now suppose we have a bialgebra filtration on $\sU$ (that is, both an algebra and coalgebra filtration). We extend it to an algebra filtration on the crossed-product $\sA \# \sU$ by putting $\sA$ in degree $0$.

Let $C \in \sU$ be a central element of order $r$. In general, we have
\[
\boldsymbol{\Delta}(C) = \sum_{j = 0}^r c_{j} \otimes c'_{r - j} \in \sum_{j = 0}^r \sU^j \otimes \sU^{r - j}.
\]
(We are using boldface $\boldsymbol{\Delta}$ for the coproduct in an attempt to distinguish it from the abstract Laplace operator of the preceding sections.)
Suppose that the $j = 0$ term takes the special form $K^r \otimes C$ for some $K \in \sU^0$, \emph{i.e.}
\begin{equation}
\label{eq:C}
\boldsymbol{\Delta}(C) = K^r \otimes C + \sum_{j = 1}^r c_j \otimes c'_{r - j}.
\end{equation}
Define a linear map $\theta : \sA \# \sU \to \sA \# \sU$ by putting
\begin{equation}
\label{eq:K-twist}
 \theta(a h) = (K \triangleright a) h,
\end{equation}
and extending linearly.

\begin{proposition}
\label{prop:Hopf}
Let $\sD = \sA\#\sU$ and $C\in\sU$ be as above. Then $[C, \sD^m]_{\theta^r} \subseteq \sD^{m+r-1}$ for all $m\in\NN$.
\end{proposition}

\begin{proof}
It suffices to check this on elements of the form $a h$ with $a \in \sA$ and $h \in \sU^m$. We compute
\[
\begin{split}
[C, a h]_{\theta^r} & = C a h - \theta^r(a h) C \\
& = \sum_{j = 0}^r (c_j \triangleright a) c'_{r - j} h - (K^r \triangleright a) h C.
\end{split}
\]
Using the assumption on the coproduct of $C$ and its centrality in $\sU$ we get
\[
\begin{split}
[C, a h]_{\theta^r} & = (K^r \triangleright a) C h + \sum_{j = 1}^r (c_j \triangleright a) c'_{r - j} h - (K^r \triangleright a) C h \\
& = \sum_{j = 1}^r (c_j \triangleright a) c'_{r - j} h.
\end{split}
\]
Since $c'_{r - j} h \in \sU^{r - j} \cdot \sU^m \subseteq \sU^{m + r - j}$ we conclude that $[C, a h]_{\theta^r} \in \mathcal{D}^{m + r - 1}$.
\end{proof}

This proposition shows that, assuming we can prove the elliptic estimates, $\sD$ will be a twisted algebra of $\DO$s with abstract Laplace operator $\sC$.  With some modification, this framework allows one to prove regularity for Kr\"ahmer's Dirac operators on the quantum projective spaces $\CP_q^n$.  But the proof requires specialized techniques from quantum groups which are rather different from the ideas presented here, so we will leave it for a separate paper.  Instead, for the reader well-versed in quantum groups, we will provide here a very brief summary of the simplest example, namely the Podle\'{s} sphere.  Compare the results of Neshveyev and Tuset \cite{NesTus:local_index_formula}.

\begin{example}
We follow the conventions of \cite{NesTus:local_index_formula} (with only slight modifications to the notation).  In particular, the generators $E,F,K$ of $\Uqk \defeq \Uq(\lie{su}_2)$, satisfy
 \begin{align*}
  & KEK^{-1} = qE, && KFK^{-1} = q^{-1}F, && [E,F] = \frac{K^2-K^{-2}}{q-q^{-1}},\\
  & \boldsymbol{\Delta}E = K \otimes E  + E \otimes K^{-1} , 
  &&  \boldsymbol{\Delta}F = K \otimes F + F \otimes K^{-1} , && \boldsymbol{\Delta}K = K \otimes K.\\
  & E^* = F, && F^* = E, && K^*=K.
 \end{align*}
 Let $\Uqt$ denote the abelian subalgebra of $\Uqk$ generated by $K$.
 
We write $\sO(K_q)$ for the algebra of polynomial functions on the compact quantum group $K_q=\SU_q(2)$.  It is a left $\Uqk$-module algebra with
 \[
  X \hit a \defeq (X,a_{(2)})a_{(1)}, \qquad X\in\Uqk,~a\in\sO(K_q).
 \]

 The Podle\'{s} sphere is the quantum homogeneous space $S^2_q = \CP^1_q = K_q/T$.  For each $k\in\half\ZZ$ we define\footnote{We are using half-integer spins, where \cite{NesTus:local_index_formula} uses integer spins.}
 \[
  \sO(\sE_{k}) := \{ \xi \in \sO(K_q) \st K\hit \xi = q^k\xi \},
 \]
 which is declared to be the section space of the spin $k$ bundle over $S^2_q$.  In particular $\sO(S^2_q) \defeq \sO(\sE_0)$.   The spinor bundle $\mathbb{S}$ is defined by $\sO(\mathbb{S}) \defeq \sO(\sE_{\half}) \oplus \sO(\sE_{-\half})$.  It is equipped with a Dirac operator:
 \[
  \Dirac \defeq \begin{pmatrix} 0 & E \\ F & 0 \end{pmatrix}.
 \]
 
 Dabrowski-Sitarz \cite{DabSit} proved that $(\sO(S^2_q), L^2(\mathbb{S}), \Dirac)$ is a spectral triple (not twisted), where $L^2(\mathbb{S})$ refers to the $L^2$-completion of $\sO(\mathbb{S})$ with respect to the Haar state of $\SU_q(2)$.  On the other hand, it is not regular as an untwisted spectral triple.  But it is regular as a \emph{twisted} spectral triple, even though the pertinent twisting is trivial on $\sA = \sO(S^2_q)$.  To prove this, we introduce a twisted algebra of differential operators.

 As described above, the algebra $\DO(K_q)\defeq\sO(K_q)\#\Uqk$ plays the role of the polynomial differential operators on $K_q$. The algebra of differential operators $\sD \defeq \DO(S^2_q;\mathbb{S})$ on $\mathbb{S}$ is the subalgebra of $2\times2$-matrices $X=(X_{ij})$ over $\DO(K_q)$ satisfying
 \[
  \begin{pmatrix}
   K&0\\0&K
  \end{pmatrix}
  \begin{pmatrix}
   X_{11}&X_{12}\\X_{21}&X_{22}
  \end{pmatrix} 
  \begin{pmatrix}
   K^{-1}&0\\0&K^{-1}
  \end{pmatrix}
  =
  \begin{pmatrix}
   X_{11}&qX_{12}\\q^{-1}X_{21}&X_{22}
  \end{pmatrix}.
 \]
 We can filter $\Uqk$ by declaring the generators $E,F$ to be order $1$, while $K,K^{-1}$ are order $0$.  This extends to a filtration of $\DO(K_q)$ where the functions $\sO(K_q)$ have order $0$, and thus to a filtration on $\sD$.

 The Casimir element for $\Uqk$ is
 \[
  C = EF + \left( \frac{q^{-\half}K^2 + q^\half K^{-2}}{q-q^{-1}}\right)^2,
 \]
which is an order two element, equal to $EF$ (or $FE$) modulo lower order.  Moreover, since
 \[
  \boldsymbol{\Delta}(EF) = K^2\otimes EF + KF\otimes EK^{-1} + EK\otimes K^{-1}F + EF \otimes K^{-2}
 \]
 we see that $\boldsymbol{\Delta}C \equiv K^2\otimes C$ modulo elements of order at most $1$ in the second leg.  In other words, $C$ satisfies Equation \eqref{eq:C}.  If we let $C$ act diagonally on sections of the spinor bundle $\sO(\mathbb{S})$, then Proposition \ref{prop:Hopf} shows that $[C,\sD^m]_{\theta^2} \subset \sD^{m+1}$ where $\theta$ is the twisting given by
 \[
  \theta(aX) = (K \hit a)X, \qquad (a\in\sO(K_q),~X\in\Uqk),
 \]
 which we extend entry-wise to $2\times2$-matrices.

 The Dirac operator $\Dirac$ satisfies $\Dirac^2=C$ as operators on $L^2(\mathbb{S})$.
 The elliptic estimates for $\sD$ with respect to $\Delta \defeq 1+\Dirac^2$ can be readily checked.  As a result, $(\sO(S^2_q), L^2(\mathbb{S}), \Dirac)$ is a regular (twisted) spectral triple.

Note that the twisting $\theta$ is only twisting the coefficient function $a$, not the constant coefficient differential operator $X$.  We are in the situation where the twisting $\theta$ is not an algebra automorphism of $\sD$, but only an algebra automorphism at the level of principal symbols. 
\end{example}


\section{Zeta functions}
\label{sec:zeta-functions}

We will conclude with some comments on the residues of zeta functions in our context, since this is one of the first major consequences of the abstract pseudodifferential calculus.  

Classically, one considers zeta functions of the form $\zeta_X(z) = \Tr(X\Delta^{-\frac{z}{r}})$, where $X\in\sD$ is an abstract differential operator and $\Tr$ is the operator trace, which is well-defined for $\Re(z)$ sufficiently large.  
Typically, the functions $\zeta_X$ extend to meromorphic functions of $z$.  But recall that for certain results one needs the additional condition that all poles of $\zeta_X$ be simple---this is the \emph{simple dimension spectrum} condition (see \cite{ConMos:local_index_formula, Higson:local_index_formula}).  For instance, this condition is needed to prove that the \emph{residue trace}
$ \tau(X) \defeq \Res_{z=0} \zeta_X(z)$ is a trace.  

The spectral triples associated to quantum homogeneous spaces, however, do not generally have simple dimension spectrum.  One potential solution is to introduce an modular operator $\rho$ (closed, unbounded, with core $H^\infty$) into the definition of the zeta function:
\begin{equation}
\label{eq:zeta}
 \zeta_X(z) \defeq \Tr(\rho X\Delta^{-\frac{z}{r}}).
\end{equation}
An example of this phenomenon is given by the Podle\'s sphere, where the introduction of an appropriate $\rho \neq 1$ gives simple poles for the zeta function, while the case $\rho = 1$ gives double poles, see \cite[Lemma 1]{KraWag}.

We will require that the twisted algebra of $\DO$s $\sD$ be stable under conjugation by $\rho$, and denote the resulting algebra automorphism by $\sigma$:
\[
 \sigma(X) \defeq \rho X \rho^{-1}.
\]
In this case, as we shall now show, the residue trace will be a twisted trace: $\tau(XY) = \tau(Y\sigma(X))$, where $\sigma(X) = \rho X \rho^{-1}$.  

Note, though, that the twisting $\sigma$ here is not necessarily related to the twisting $\theta$ used above.  For instance, the following proposition holds for the standard residue trace (with $\rho=1$) of a twisted algebra $\sD$ (with $\theta\neq\id$), provided that the standard zeta function does indeed extend meromorphically with only simple poles.

\begin{proposition}
Let $\mathcal{D}$ be a twisted algebra of $\DO$s, with diagonalizable twisting $\theta$. Suppose that for all $X \in \sD$ we have:
\begin{itemize}
\item $\rho X \Delta^{-\frac{z}{r}}$ is trace-class for all $z \in \mathbb{C}$ with sufficiently large real part,
\item $\zeta_X(z)$ extends to a meromorphic function on $\mathbb{C}$ with only simple poles.
\end{itemize}
Then the functional $\Phi: \mathcal{D} \to \mathbb{C}$ defined by
\[
\Phi(X) = \mathrm{Res}_{z = 0} \, \zeta_X(z)
\]
is a $\sigma$-twisted trace, that is $\Phi(X Y) = \Phi(Y \sigma(X))$ for all $X, Y \in \sD$.
\end{proposition}

\begin{proof}
Using $\sigma(X) = \rho X \rho^{-1}$ and the trace property we write
\[
\zeta_{X Y}(z) = \mathrm{Tr}(\rho X Y \Delta^{-\frac{z}{r}}) = \mathrm{Tr}(\rho Y \Delta^{-\frac{z}{r}} \sigma(X)).
\]
We can apply the expansion of Proposition \ref{prop:Delta_expansion} to $\Delta^{-\frac{z}{r}} \sigma(X)$. Let us denote by $m$ the order of $\sigma(X) \in \sD$.
Then we get
\[
\begin{split}
\zeta_{X Y}(z) & = \mathrm{Tr}(\rho Y \theta^{-z}(\sigma(X)) \Delta^{-\frac{z}{r}}) \\
& + \sum_{k = 1}^n \sum_{\bfmu\in W(k)} \binom{-z/r}{k}_{\!\!\bfmu} \mathrm{Tr}(\rho Y \nabla^\bfmu(\sigma(X)) \Delta^{-\frac{z}{r} - k}) \\
& + \mathrm{Tr}(\rho Y Q),
\end{split}
\]
where $Q \in \mathrm{Op}^{m - n - 1}$.

Now we take the residue at $z = 0$ of the function $\zeta_{X Y}(z)$. The second line vanishes, since by assumption the zeta function has only simple poles and the coefficients $\binom{-z/r}{k}_{\!\bfmu}$ vanish at $z = 0$ for all $k>0$ (see Proposition \ref{prop:mu-binomial}).
Similarly the third line can be made to vanish by taking $n$ large enough, since $\mathrm{Tr}(\rho Y Q)$ becomes holomorphic at $0$ in this case.
Therefore we are left with
\[
\Phi(X Y) = \mathrm{Res}_{z = 0}\, \mathrm{Tr}(\rho Y \theta^{- z}(\sigma(X)) \Delta^{-\frac{z}{r}}).
\]

The twist $\theta^{-z}$ in this expression can be removed, as we now argue. Suppose to begin with that $X$ is homogeneous of weight $\lambda$, that is $\theta^{-z}(X) = \lambda^{-z} X$. The factor $\lambda^{-z}$ can be pulled out of the trace and evaluating it at $z = 0$ gives $1$. For general $X$, it suffices to decompose $\sigma(X)$ into homogeneous components.  

Therefore, we get
\[
\Phi(X Y) = \mathrm{Res}_{z = 0}\, \mathrm{Tr}(\rho Y \sigma(X) \Delta^{-\frac{z}{r}}) = \Phi(Y \sigma(X)).
\]
This concludes the proof.
\end{proof}

\bigskip

\appendix
\section{Appendix: Quantum analogues of the Cauchy Differentiation Formula}
\label{sec:mu-Cauchy}

\subsection{A generalized Cauchy Differentiation Formula}
Recall that the classical Cauchy Differentiation Formula is
\[
  \frac{1}{2\pi i} \int_\Gamma \frac{f(\lambda) }{(\lambda-t)^{n+1}}\,  d\lambda
    = \frac{f^{(n)}(t)}{n!}.
\]
To state the generalized formula, we need to introduce differential-difference operators of the following type.

\begin{definition}
 Let $\bfmu = (\mu_0,\ldots,\mu_n) \in \CC^{n+1}$.  We write $\mathrm{mult}(\mu_i)$ for the multiplicity of $\mu_i$ in $\bfmu$, \emph{i.e.}, 
 $
  \mathrm{mult}(\mu_i) = \# \{ j \st \mu_j=\mu_i \}.
 $
 
 If $f(s)$ is a holomorphic function of $s\in\CC$ (or some appropriate open subset), we write $p_{f, \bfmu}(s)$ for the polynomial of degree $n$ which agrees with $f(s)$ at each $s=\mu_i$ to order $\mathrm{mult}(\mu_i)$.  The \emph{$\bfmu$-derivative} of a holomorphic function $f$ at $t\in\CC$ is then defined as
 \[
  \dmu f(t) := p_{f,t\bfmu}^{(n)} \in \CC,
 \]
 where we note that the $n$th derivative $p_{f,t\bfmu}^{(n)}(s)$ is polynomial of order $0$.
\end{definition}

For instance, the ordinary $n$th derivative is equal to $\partial_{\bfmu}$ when $\bfmu = (1,\ldots,1) \in \CC^{n+1}$.  
On the other hand, if all the $\mu_i$ are distinct we obtain a higher order difference operator.  For instance, $\partial_{(a,b)} f (t) = \frac{f(at)-f(bt)}{a-b}$ when $a\neq b$.  In particular $\partial_{(1,q)}$ is the $q$-derivative $D_q$ when $q\neq 1$ (see \cite[\S2.2.1]{KliSch}).

\begin{proposition}
\label{prop:mu-Cauchy}
 Let $f$ be a holomorphic function on a simply connected domain $U\subseteq\CC$, and let $\Gamma$ be a simple contour in $U$.  Let $t\in\CC$ and let $\bfmu \in \CC^n$ such that $\mu_i t$ lies in the interior of $\Gamma$ for each $i$.  Then
 \begin{equation}
 \label{eq:mu-Cauchy}
  \frac{1}{2\pi i} \int_\Gamma f(\lambda) \,{\textstyle \prod_{i=0}^n (\lambda - \mu_it)^{-1} } \, d\lambda
    = \frac{\dmu f(t)}{n!}.
 \end{equation}
\end{proposition}

\begin{proof}

It suffices to prove this formula on the dense open subset of $\bfmu$ for which all coefficients $\mu_i$ are distinct, since both sides of the formula are holomorphic functions of $\bfmu\in\CC^{n+1}$ so long as $\mu_it$ lies inside $\Gamma$ for all $i$.  

\begin{lemma}
 \label{lem:partial_fractions}
 If $a_0,\ldots,a_n$ are distinct complex numbers then
 \[
   \prod_{i=0}^n (\lambda-a_i)^{-1} 
    = \sum_{i=0}^n \left( (\lambda-a_i)^{-1} \prod_{j\neq i} (a_i-a_j)^{-1} \right)
 \]
 as rational functions.
\end{lemma}

\begin{proof}
 Consider the following polynomial in $\lambda$:
  \[
   g(\lambda) \defeq \sum_{i=0}^n \left( \prod_{j\neq i} \frac{(\lambda-a_j)}{(a_i-a_j)} \right).
  \]
  It satisfies $g(a_k)=1$ for all $k=0,\ldots,n$.  Since $g$ has order $n$, it follows that $g=1$.  Multiplying $g$ by $\prod_{i=0}^n (\lambda-a_i)^{-1}$ gives the result.    
\end{proof}

Applying this lemma, the left-hand side of Equation \eqref{eq:mu-Cauchy} becomes
\begin{align}
  \frac{1}{2\pi i} \int_\Gamma f(\lambda) & \,{\textstyle \prod_{i=0}^n (\lambda - \mu_it)^{-1} } \, d\lambda 
   \nonumber \\
   &= \sum_{i=0}^n \left( {\textstyle \prod_{j\neq i} (\mu_it-\mu_jt)^{-1}}\right) 
     \frac{1}{2\pi i}  \int_\Gamma f(\lambda) (\lambda-\mu_i t)^{-1} \,d\lambda 
   \nonumber \\
   &= \sum_{i=0}^n f(\mu_it) \,{\textstyle \prod_{j\neq i} (\mu_it-\mu_jt)^{-1}}.
   \label{eq:explicit_Cauchy}
\end{align}

For the right-hand side, we recall that
\[
 \dmu f(t) = p_{f,t\bfmu}^{(n)}
\]
where $p_{f,t\bfmu}$ is the unique polynomial of degree $n$ which agrees with $f$ at each $\mu_i t$.  Specifically,
\[
 p_{f,t\bfmu}(s) = \sum_{i=0}^n f(\mu_i t) \prod_{j\neq i} \frac{(s-\mu_j t)}{(\mu_i t - \mu_j t)}.
\]
Considering the coefficient of $s^n$ in this polynomial gives
\begin{equation}
\label{eq:explicit_q-derivative}
 \frac{\dmu f(t)}{n!} = \frac{p_{f,t\bfmu}^{(n)}}{n!} =  \sum_{i=0}^n f(\mu_i t) \prod_{j\neq i} \frac{1}{(\mu_it-\mu_jt)}.
\end{equation}
Comparing \eqref{eq:explicit_Cauchy} and \eqref{eq:explicit_q-derivative} proves the proposition.
\end{proof}

\subsection{$\bfmu$-binomial coefficients}
\label{sec:mu-binomial}

We will apply the above formula to the functions $f(t) = t^z$ where $z\in\CC$ (with branch cut for $t$ on the negative real axis).  
Let $\bfmu \in \RR_+^{n+1}$ be an $n$-tuple of strictly positive reals.  Then the $\bfmu$-derivative of $t^z$ is well-defined for all $t\in\CC\setminus(-\infty,0]$, and from Equation \eqref{eq:explicit_q-derivative} we deduce that it is a constant multiple of $t^{z-n}$.  The constant will be called the \emph{$\mu$-binomial coefficient} and denoted $\binom{z}{n}_{\!\bfmu}$, \emph{i.e.},
\begin{equation}
 \label{eq:mu-binomial}
 \dmu (t^z) = \binom{z}{n}_{\!\!\bfmu} t^{z-n} .
\end{equation}

The next proposition summarizes some of the basic properties of the $\bfmu$-binomial coefficients.  

\begin{proposition}
 \label{prop:mu-binomial}
 Let $\bfmu = (\mu_0,\ldots,\mu_n)\in\RR_+^{n+1}$ and $z\in\CC$.  
 \begin{enumerate}
  \item The function $z \mapsto \binom{z}{n}_{\!\bfmu}$ is entire.
  \item The following analogue of Pascal's identity holds:
  \begin{equation}
   \label{eq:Pascal}
   \binom{z+1}{n}_{\!\!\bfmu} = \mu_0 \binom{z}{n}_{\!\!\bfmu} + \binom{z}{n-1}_{\!\!\check{\bfmu}},
  \end{equation}
  where $\check{\bfmu} = (\mu_1,\ldots,\mu_n)\in\RR_+^n$ is the $n$-tuple obtained by removing $\mu_0$.
  \item If $\bfmu = (1,\ldots,1)$ then $\binom{z}{n}_{\!\bfmu} = \frac{z(z-1)\cdots(z-n+1)}{n!}$ is the standard binomial coefficient.
  \item If $\bfmu = (1,q,\ldots,q^n)$ for $q\neq 1$, then $\binom{z}{n}_{\!\bfmu} = \frac{(1-q^z)(1-q^{z-1})\cdots(1-q^{z-n+1})}{(1-q)(1-q^2)\cdots(1-q^{n})}$ is the Gaussian binomial coefficient.
  \item If $n=0$, so that $\bfmu = (\mu)$ for some $\mu\in\RR_+$, then $\binom{z}{0}_{\!\bfmu} = \mu^z$.
  \item If $n>0$, then $\binom{0}{n}_{\!\bfmu}=0$.
 \end{enumerate}
\end{proposition}

\begin{proof}
 Claim (1) follows from the holomorphicity of $z\mapsto t^z$.  For (2), we use Equation \eqref{eq:explicit_q-derivative} to obtain an explicit formula for the $\bfmu$-binomial coefficients when all $\mu_i$ are distinct:
 \begin{equation}
 \label{eq:mu-binomial_explicit}
  \binom{z}{n}_{\!\!\bfmu} = \sum_{i=0}^n \frac{\mu_i^z}{ \prod_{j\neq i} (\mu_i-\mu_j) }.
 \end{equation}
 In this case,
 \begin{align*}
  \binom{z+1}{n}_{\!\!\bfmu} & = \sum_{i=0}^n \frac{\mu_i^{z+1}}{ \prod_{j\neq i} (\mu_i-\mu_j) } \\
   &= \sum_{i=0}^n \frac{\mu_0 \mu_i^{z}}{ \prod_{j\neq i} (\mu_i-\mu_j) } + \sum_{i=0}^n \frac{\mu_i^{z}(\mu_i-\mu_0)}{ \prod_{j\neq i} (\mu_i-\mu_j) }  \\
   &= \mu_0 \sum_{i=0}^n \frac{ \mu_i^{z}}{ \prod_{j\neq i} (\mu_i-\mu_j) } + \sum_{i=1}^n \frac{\mu_i^z}{ \prod_{0\neq j\neq i} (\mu_i-\mu_j) }  \\
   &= \mu_0 \binom{z}{n}_{\!\!\bfmu} + \binom{z}{n-1}_{\!\!\check{\bfmu}}.
 \end{align*}
 The general case follows by continuity in $\bfmu$.  Claim (3) is standard since $\partial_{(1,\ldots,1)} = \frac{d^n}{dt^n}$.  Claim (4) follows by comparing Pascal's Identity for the $q$-binomial coefficients (see, \emph{e.g.}, \cite[\S2.1.2]{KliSch})---we skip the details since in any case we don't need this.  The final two claims are easy calculations.
\end{proof}

The following is an immediate consequence of the generalized Cauchy Differentiation Formula (Proposition \ref{prop:mu-Cauchy}).

\begin{corollary}
\label{cor:mu-Cauchy}
  Let $\bfmu = (\mu_0,\ldots,\mu_n)\in\RR_+^{n+1}$.  If $t\in\RR_+$, then for any $z$ with $\Re(z)<0$ and any vertical contour $\Gamma$ separating $t$ from $0$, we get
  \[
    \frac{1}{2\pi i} \int_\Gamma \lambda^z (\lambda - \mu_0 t)^{-1} \cdots (\lambda - \mu_n t)^{-1} \, d\lambda 
     = \binom{z}{n}_{\!\!\bfmu} t^{z-n}.
  \]
\end{corollary}

Lemma \ref{lem:mu-Cauchy-Delta} now follows from Corollary \ref{cor:mu-Cauchy} by the holomorphic functional calculus.

\bibliographystyle{alpha}
\bibliography{refs}

\end{document}